\documentclass[a4paper,abstracton]{scrartcl}
\usepackage[utf8]{inputenc}
\usepackage{amsfonts,amsthm,amssymb,amsmath}
\usepackage[inline]{enumitem}
\usepackage{color}
\usepackage{graphicx}
\usepackage{authblk}

\newtheorem{thm}{Theorem}

\newtheorem{conj}[thm]{Conjecture}

\theoremstyle{remark}

\newtheorem{clm}{Claim}

\DeclareMathOperator{\Aut}{Aut}

\title{A bound for the distinguishing index of regular graphs}
\author[1]{Florian Lehner\thanks{Florian Lehner was supported by the Austrian Science Fund (FWF), grant J 3850-N32}}
\author[2]{Monika Pil{\'s}niak\thanks{This work was partially supported by Ministry of Science and Higher Education of Poland and OEAD grant no.PL 08/2017.}}
\author[2]{Marcin Stawiski}
\affil[1]{Institute of Discrete Mathematics, Graz University of Technology, \protect\\ Steyrergasse 30, 8010 Graz, Austria}
\affil[2]{AGH University, Department of Discrete Mathematics, \protect\\al. Mickiewicza 30, 30-059 Krakow, Poland}

\begin{document}
\maketitle
\begin{abstract}
An edge-colouring of a graph is distinguishing if the only automorphism that preserves the colouring is the identity.
It has been conjectured that all but finitely many connected, finite, regular graphs admit a distinguishing edge-colouring with two colours. We show that all such graphs except $K_2$ admit a distinguishing edge-colouring with three colours.
This result also extends to infinite, locally finite graphs.
Furthermore, we are able to show that there are arbitrary large infinite cardinals $\kappa$ such that every connected $\kappa$-regular graph has a distinguishing edge-colouring with two colours.
\end{abstract}

\section{Introduction }

Let $G$ be a connected, finite or infinite graph and let $\Aut(G)$ denote its group of automorphisms. The \emph{distinguishing index} of $G$, denoted by $D'(G)$, is the least number of colours needed to colour the edges of $G$  such that the only colour preserving automorphism is the identity. This concept was first introduced in \cite{KP}; the analogous concept for vertex colouring, often denoted by $D(G)$ is significantly older, see \cite{AC,BAB}. Note that $D(G)$ and $D'(G)$ can be arbitrarily far apart. As an example, it is easy to see that the complete graph $K_n$ satisfies $D(K_n) = n$, but $D'(K_n) = 2$ for $n \geq 6$. Another example is the complete bipartite graph. On the other hand, it is known that for most graphs $D'(G)$ does not exceed $D(G)$, see \cite{LS}.

For connected graphs with finite maximum degree $\Delta$, it is known that $D'(G) \leq \Delta$, unless $G$ is $C_3$, $C_4$ or $C_5$. This bound is sharp, and the graphs which attain it are fully characterised, see \cite{P,PS}.
On the other hand, there are many graph classes where better bounds are possible. For example it is known that, apart from finitely many exceptions, $D'(G) \leq 2$ for all traceable graphs \cite{P}, $3$-connected planar graphs \cite{PT},   Cartesian powers of finite and countable graphs \cite{BP2, GKP}, and countable graphs where every non-trivial automorphism moves infinitely many edges \cite{lehner-edge}.

The following conjecture made by Alikhani and Soltani, and independently by Imrich, Kalinowski, Pil{\'s}niak and Wo{\'z}niak \cite{IKPW}, would imply improved bounds in the case when the minimum and the maximum degree of the graph are not too far from each other.

\begin{conj} \cite{IKPW}
\label{conjecture-general}
If $G$ is a connected finite graph with minimum degree $\delta\geq 2$ and maximum degree $\Delta$, then
$D'(G)\leq \lceil\sqrt[\delta]{\Delta}\rceil +1$,
with equality only for $G=K_{\delta,r^{\delta}}$ and finitely many graphs on at most $6$ vertices.
\end{conj}

We observe that the second part of Conjecture~\ref{conjecture-general} is not true in general; equality also holds for other complete bipartite graphs. Indeed, consider a complete bipartite graph $K_{p,q}$ with $p<q$ and an integer $d>1$ such that $d^p - \lceil \log_d p\rceil + 1 \leq q\leq d^p$. It is known that in this case $D'(K_{p,q})=d+1=\lceil\sqrt[\delta]{\Delta}\rceil +1$, see~\cite{FI,IJK}.

It is still possible, however, that the first part of Conjecture~\ref{conjecture-general} holds, that is, $D'(G)\leq \lceil\sqrt[\delta]{\Delta}\rceil +1$ for any finite connected graph $G$ with minimum degree $\delta\geq 2$ and maximum degree $\Delta$. Note that if this is the case, then the bound is almost tight for $K_{\delta, \Delta}$. Indeed, assume that we have a colouring of the edges of $K_{\delta, \Delta}$ with $r \leq \sqrt[\delta]{\Delta}$ colours. Let $v_1, \dots, v_\delta$ be an enumeration of the vertices in the part of size $\delta$. Assign to each vertex $u$ in the part of size $\Delta$ the sequence of colours of the edges $uv_1, \dots, uv_\delta$. If two of these sequences coincide, then the corresponding vertices can be swapped by  a colour preserving automorphism. This can only be avoided if $r = \sqrt[\delta]{\Delta}$ and each of the $\Delta = r^\delta$ different sequences appears exactly once. It is easy to check that for such a colouring any permutation of $v_1, \dots, v_\delta$ extends to a colour preserving automorphism of $K_{\delta, \Delta}$, thus showing that $D'(K_{\delta, \Delta}) > \sqrt[\delta]{\Delta}$.

In this paper we are interested in regular graphs, that is, graphs where $\delta = \Delta$. We prove that the first part of Conjecture~\ref{conjecture-general} holds for such graphs. However, based on the observation that $D'(K_{\Delta,\Delta}) = 2$ for $\Delta \geq 4$, we conjecture that our result is not tight.
%Conjecture~\ref{conjecture-general} boils down to the following; note that equality in Conjecture \ref{conjecture-general} cannot hold for regular graphs on more than $6$ vertices.

\begin{conj}
\label{conjecture-reg}
If $G$ is a connected, finite, regular graph, then $D'(G)\leq 2$, unless $G$ is either $K_n$ for $n \leq 5$, or $K_{n,n}$ for $n \leq 3$, or $C_5$.
\end{conj}

Besides Conjecture \ref{conjecture-general}, there is another compelling reason why one might expect this to be true. It  is easy to show that if a graph $G$ with at least $7$ vertices contains a Hamiltonian path, then $D'(G) = 2$.  Indeed, for such a graph one can obtain an asymmetric spanning tree $H$ with one vertex of valence three and three branches of different lengths simply by adding an edge to $P$ and then removing an edge. Colouring all edges of such a spanning tree with one colour and the remaining edges with a second colour yields a distinguishing edge $2$-colouring; for more details see \cite{P}.

While there are regular graphs that do not admit Hamiltonian paths, it is not unreasonable to think that most of them at least have an asymmetric spanning tree. By Dirac's Theorem, a graph on $n$ vertices whose minimum degree is at least $\frac n2$ has a Hamiltonian cycle. Consequently, if $G$ is a regular graph, then either $G$ or its complement has a Hamiltonian cycle, and one might hope that a Hamiltonian cycle in the complement of $G$ can also be used to construct an asymmetric spanning tree of $G$. For vertex transitive graphs it is also worth pointing out that Lov\'asz~\cite{LL}  conjectured that all of them admit a Hamiltonian path.

As mentioned above, in this paper we make progress on Conjecture~\ref{conjecture-reg} and some natural generalisations of it.
More precisely, in Section \ref{main} we prove the following theorem.
It is worth pointing out that this result covers both finite and infinite, locally finite graphs.
\begin{thm}
\label{thm:main}
Let $G$ be connected $\Delta$-regular graph for $\Delta \in \mathbb N \setminus \{1\}$.
Then $D'(G)\leq 3$.
\end{thm}

Moreover, in Section \ref{infinite} we consider $\Delta$-regular graphs where $\Delta$ is infinite. It was shown in \cite{BP} that any connected graph in which all degrees are countable satisfies $D'(G) \leq 2$. We extend this result by showing that there exist arbitrary large cardinals $\kappa$ such that every connected $\kappa$-regular graph has distinguishing index at most two. Namely, every fixed point of the aleph hierarchy has the said property.

%%%%%%%%%%%%%%%%%%%%%%%%%%%%%%%%%%%%%%%%%%%%%%%%%%%%%%%%%
%%%%%%%%%%%%%%%%%%%%%%%%%%%%%%%%%%%%%%%%%%%%%%%%%%%%%%%%%
%\section{Preliminaries}
\section{Proof of the main result}\label{main}

In this section we prove the following theorem which easily implies Theorem \ref{thm:main}.

\begin{thm}
Let $G$ be connected $\Delta$-regular graph for some finite $\Delta$.
Unless $G= K_2$, there is a distinguishing edge colouring with $3$ colours---red, green, and blue---that additionally satisfies the following property.
\begin{enumerate}[label=(\textasteriskcentered)]
    \item \label{itm:star}
    There is at most one vertex all of whose incident edges are coloured blue. If $G = K_n$, then there is no such vertex.
\end{enumerate}
\end{thm}

\begin{proof}
It is not hard to see (c.f.\ \cite{KP}) that there is such a colouring for complete graphs on at least $3$ vertices.
Furthermore, by results from \cite{P} and \cite{PS}, it is known that for any graph $G$ with maximum degree $\Delta$ it holds that $D'(G)\leq \max \{ \Delta-1, 3 \}$; the proofs also show that if $G$ is regular of degree $\Delta \leq 4$, then there is a distinguishing edge colouring satisfying \ref{itm:star}. Therefore, we may assume that $G$ is not complete, that $\Delta \geq 5$, and that any $\Delta'$-regular graph with $\Delta'<\Delta$ has a distinguishing edge colouring with $3$ colours which satisfies \ref{itm:star}.

For the rest of the proof, fix an arbitrary root vertex $r$ in $G$.
Let $\mathcal S$ be the set of orbits under $\Aut(G,r)$, where $\Aut(G,r) = \{\gamma \in \Aut(G) \mid \gamma r = r\}$ is the stabiliser of $r$ in $\Aut G$.
Order $\mathcal S$ by distance from $r$, with ties broken arbitrarily, and let $(S_i)_{0 \leq i < |\mathcal S|}$ be the corresponding enumeration.
%Clearly, $S_0 =  \{r\}$.
Denote by $E_i$ the set of edges incident to $S_i$.
For $i< |\mathcal S|$, let $m(i) = \max \{j \mid \exists k \leq i \colon E_k \cap E_j \neq \emptyset\}$, in other words, $m(i)$ is the maximal $j$ such that there is an edge connecting $S_k$ to $S_j$ for $k \leq i$. Note that if $i<|\mathcal{S}|$, then $m(i)\geq i+1$ because there must be an edge from $S_{i+1}$ to $S_k$ for some $k\leq i$.

Let $\hat E_i = E \setminus \bigcup_{j > m(i)} E_j$, that is, $\hat E_i$ is the set of edges both of whose endpoints are in $\bigcup _{j \leq m(i)} S_j$.
Clearly, $E_i \subseteq \hat E_i$.
It furthermore follows from the definition of $m(i)$ that $\hat E_i \supseteq \hat E_{i-1}$.
Note that $\Delta$ is finite, so the sets $S_i$, $E_i$ and $\hat E_i$ are finite as well.
We partition $E_i$ into the following three sets.\begin{align*}
    & \{xy\in E(G):  x\in S_i, y\in S_j \text{ and } j < i  \},\\
    & \{xy\in E(G):  x\in S_i, y\in S_j \text{ and } j > i  \},\\
    & \{xy\in E(G): x, y\in S_i  \}.
\end{align*}
We refer to edges in the first set as \emph{back edges}, to edges in the second set as \emph{forward edges}, and to edges in the third set as \emph{horizontal edges} of $S_i$ respectively.
Vertices connected to $S_i$ by forward or back edges will be referred to as \emph{forward neighbours} or \emph{back neighbours} of $S_i$ respectively.
Note that by definition of $S_i$, every vertex in $S_i$ is incident to the same number $f$ of forward edges, the same number $b$ of back edges, and the same number $h$ of horizontal edges (and clearly $f+b+h=\Delta$).
Further note that by definition, $\bigcup_{j < i} S_j$ contains all vertices that are closer to $r$ than $S_i$ (and perhaps some vertices at the same distance to $r$).
Thus for $i > 0$ every vertex in $S_i$ has at least one back neighbour; in particular $b > 0$.
%%[picture]:difference E_j and \hat E_j

Next, for $0 \leq i < |\mathcal S|$, we define a colouring $c_i$ with colours red, green, and blue, such that the following properties are satisfied.
\begin{enumerate}[label = (\Roman*)]
    \item \label{itm:rootcolour}
    $r$ is the only vertex where all incident edges are coloured blue in $c_i$,
    \item  \label{itm:changecolour}
    for $i > 0$, the colourings $c_{i-1}$ and $c_{i}$ only differ on $E_i$, and
    \item \label{itm:preservecolour}
    for $0 \leq j \leq i$, if $\gamma \in \Aut(G,r)$ preserves the restriction of $c_i$ to $\hat E_j$, then $\gamma$ pointwise fixes $S_j$.
\end{enumerate}
Before we proceed with the construction, we demonstrate how this yields a distinguishing edge colouring.

If $\mathcal S$ is finite, then let $c = c_{|\mathcal S|-1}$.
By \ref{itm:rootcolour}, any automorphism preserving $c$ must fix the root $r$, and by \ref{itm:preservecolour} any such automorphism must pointwise fix every $S_j$.
Any vertex is contained in some $S_j$, hence we conclude that $c$ is distinguishing.

Now assume that $\mathcal S$ is infinite.
Let $c$ be the pointwise limit of the $c_i$, in other words, $c(e)$ takes the same value as all but finitely many $c_i(e)$.
This exists by \ref{itm:changecolour}: if $e$ connects vertices in $S_i$ and $S_j$, then it has the same colour in every $c_k$ for $k \geq \max(i,j)$.
It also follows from \ref{itm:changecolour} that for any finite set $E' \subseteq E$ there is an index $k$ such that $c(e) = c_k(e)$ for all $e \in E'$.
Consequently, since $G$ is locally finite, $r$ is the unique vertex where $c$ assigns blue to all incident edges.
Thus $r$ must be fixed by any automorphism preserving $c$.
Since $\hat E_j$ is finite for every $j$, there must be some $k\geq j$ such that $c$ and $c_k$ agree on $\hat E_{j}$.
By \ref{itm:preservecolour}, applied to $c_k$, we conclude that any element of $\Aut(G,r)$ which preserves $c$ must fix $S_j$ pointwise.
This holds for every $j$, and each vertex is contained in some $S_j$, hence $c$ is distinguishing.

It remains to construct the colourings $c_i$. For the inductive construction, it will be useful to ensure that the colourings also satisfy the following technical conditions.
\begin{enumerate}[resume,label = (\Roman*)]
    \item \label{itm:allgreen}
    If $e \notin E_j$ for any $j \leq i$, then $c_i(e)$ is green.
    \item \label{itm:blueedges}
    If $e \in E_j$ for $j > i$ and $c_i(e) = \text{blue}$, then $e \in E_0$. In other words, the only blue edges with respect to $c_i$ are between two vertices in $\bigcup_{j \leq i} S_j$, or are incident to $r$.
\end{enumerate}

Let $c_0$ be the colouring where edges incident to $r$ are coloured blue and all other edges are coloured green. This trivially satisfies \ref{itm:rootcolour} and \ref{itm:preservecolour}--\ref{itm:blueedges}. Property \ref{itm:changecolour} does not apply for $i=0$.

Now let $i > 0$ and assume that $c_j$ has already been defined for every $j < i$. By \ref{itm:changecolour}, in order to describe $c_i$ it suffices to describe $c_i$ on $E_i$. The construction of the colouring $c_i$ consists of two steps: in the first step we colour the horizontal edges, in the second step we use a subtle recolouring procedure on the forward and back edges to get rid of any remaining symmetries.

The first part is fairly straightforward.
Let $H$ be the subgraph of $G$ induced by the horizontal edges of $S_i$.
\begin{enumerate}[label=(H\arabic*)]
    \item \label{itm:cross0}
    If $h=0$, then there are no horizontal edges to colour.
    \item \label{itm:cross1}
    If $h=1$, then colour the edges with red, green, and blue such that no orbit containing at least two edges with respect to the pointwise stabiliser of $\bigcup_{j < i} S_j$ has more than half of its edges coloured by the same colour.
    \item \label{itm:cross2}
    Otherwise $2 \leq h < \Delta$, so each component $K$ of $H$ is a connected $\Delta'$-regular graph with $2 \leq \Delta' < \Delta$.
    By our general induction assumption we can find an edge colouring of $K$ which satisfies \ref{itm:star} such that any colour preserving automorphism which setwise fixes $K$ must also fix it pointwise.
\end{enumerate}

Note that the colouring of the horizontal edges defined above generally will not break all automorphisms acting non-trivially on $S_i$.
In case $h=0$ we did not break any automorphisms at all, but even if $h \geq 2$ there could be colour preserving automorphisms which permute the components of .
We call an automorphism \emph{persistent} if it pointwise fixes every $S_j$ for $j < i$ and preserves the colouring on the horizontal edges defined above.

The second part hinges on the following recolouring procedure for forward and back edges.
A \emph{decoration} of a component $K$ of $H$ is a pair $(F,B)$, with the following properties.
\begin{enumerate}[label = (D\arabic*)]
    \item \label{itm:forward}
    $F$ is a set of forward edges incident to $K$.
    \item \label{itm:back}
    $B$ is either the empty set, or a set consisting of a single red back edge incident to $K$, or a set consisting of two green back edges incident to $K$.
\end{enumerate}
We say that we \emph{decorate} $K$ by $(F,B)$, if we recolour the forward edges incident to $K$ such that exactly those in $F$ are red and the rest are green, and change the colour of the back edges in $B$ to blue while all other back edges incident to $K$ are coloured with the same colour as in $c_{i-1}$.

Call a decoration $(F,B)$ \emph{asymmetric}, if any persistent automorphism that fixes $K$ setwise and maps $F$ and $B$ onto themselves must fix $K$ pointwise.
Call two decorations \emph{similar}, and write $(F,B) \sim (F',B')$, if there is a persistent automorphism mapping $F$ to $F'$ and  $B$ to $B'$.
Clearly, $\sim$ is an equivalence relation on the set of all decorations.
By definition of $\sim$, if $(F,B) \nsim (F',B')$ and $K$ and $K'$ are decorated by $(F,B)$ and $(F',B')$ respectively, then no colour preserving, persistent automorphism can map $K$ to $K'$.

Our strategy in the second part is to assign asymmetric decorations to components such that no two decorations are similar, and then decorate every component with the corresponding decoration.

For this purpose it is enough to show the following claim.

\begin{clm}
\label{clm:DgeqN}
Let $D_K$ be the number of non-similar asymmetric decorations available at some component $K$, and let $N_K$ be the number of components that $K$ can be mapped to by persistent automorphisms. Then $D_K \geq N_K$.
\end{clm}
Indeed, if this is true, then we can greedily assign decorations to components, making sure that each of them receives a decoration that is not similar to any decorations used on other components in this orbit.

To bound $N_K$ from above, note that every vertex is incident to at least one back edge. Any persistent automorphism must fix the other endpoint of this back edge, so $N_K \leq \Delta$. Moreover, if this bound is sharp, then it is only sharp for $S_1$ (and this can only happen if $S_1$ contains all neighbours of $r$), otherwise the bound decreases to $N_K \leq \Delta -1$. Further recall that persistent automorphisms are required to preserve the colouring of the horizontal edges. Consequently, if $h=1$, then our colouring of the horizontal edges ensures that $N_K \leq \lfloor \Delta / 2 \rfloor$.

Next we establish lower bounds for $D_K$.
For this purpose, we construct sets of non-similar decorations as follows.
Let $F^*$ be some set of forward edges incident to $K$.
Let $\mathcal F$ be a set of subsets of $F^*$ all of which have different cardinalities $0,\dots,|F^*|$.
Let $B^*$ be some set of backward edges incident to $K$ not containing any blue edges chosen such that no persistent automorphism moves one element of $B^*$ to another.
Let $\mathcal B$ be the set of subsets of $B^*$ that satisfy \ref{itm:back}.

It is easy to see that any two members $(F,B) \neq (F',B')$ of $\mathcal F \times \mathcal B$ are non-similar.
Indeed, if $F \neq F'$, then $|F| \neq |F'|$ and there is no automorphism moving $F$ to $F'$.
If $B \neq B'$ then by the condition on $B^*$ there cannot be a persistent automorphism mapping $B$ to $B'$.
In particular $D_K$ will be at least the number of asymmetric members of $\mathcal F \times \mathcal B$.

It will be convenient to have a lower bound on the size of $\mathcal F \times \mathcal B$.
Clearly, $|\mathcal F| = |F^*|+1$ and $|\mathcal B| \geq 1$ because $\emptyset \in \mathcal B$.
If $B^*\neq \emptyset$, then $|\mathcal B| = 1 + b_r + \binom {b_g}2$, where $b_r$ and $b_g$ are the number of red and green edges in $B^*$ respectively.
As $\binom {b_g}2 \geq b_g - 1$ and $|B^*| = b_r + b_g$, we have $|\mathcal B| \geq \max(1,|B^*|)$, and consequently
\[
|\mathcal F \times \mathcal B| \geq (|F^*|+1) \cdot \max(1,|B^*|) \geq |F^*|+ \max(1,|B^*|).
\]

The choices of $F^*$ and $B^*$ depend on the number  of horizontal edges incident to each vertex, we distinguish cases $h=0$, $h=1$, and $h \geq 2$.

First assume that $h = 0$.
Then $K$ consists of a single vertex $x$ incident to $f$ forward edges and $b$ back edges.
Let $F^*$ consist of all forward edges incident to $x$, and let $B^*$ consist of all non-blue back edges incident to $x$.

Note that any persistent automorphism must fix all back neighbours of $x$, whence there cannot be a persistent automorphism mapping any edge in $B^*$ to a different one.
Moreover, any decoration of $K$ is asymmetric, since $K$ only consists of a single vertex. Thus $D_K\geq |\mathcal F \times \mathcal B|$.

Clearly, $|F^*| = f$ and $|B^*|\geq b-1$ because there is at most one back edge connecting $x$ to $r$.
If $b = 1$, then we get
\[D_K \geq |\mathcal F \times \mathcal B| \geq f+1 = \Delta \geq N_K.\]
If $b > 1$, then $r$ is not the only back neighbour of $x$.
Hence $i > 1$, and in particular $N_K \leq \Delta  - 1$.
This means that we get
\[D_K \geq |\mathcal F \times \mathcal B| \geq f + b - 1 = \Delta - 1 \geq N_K,\]
thus completing the proof of Claim~\ref{clm:DgeqN} for the case $h=0$.

If $h = 1$, then $K$ consists of two vertices connected by a single edge.
Let $x$ be one of the two vertices, and let $F^*$ consist of all forward edges incident to $x$, and let $B^*$ consist of all non-blue back edges incident to $x$.

Note that the only way a persistent automorphism could fix $K$ setwise, but not pointwise, is by swapping its two vertices.
As in the previous case, any persistent automorphism must fix all back neighbours of $x$ and thus there cannot be a persistent automorphism mapping any edge in $B^*$ to a different one.
Moreover, a decoration $(F,B) \in \mathcal F \times \mathcal B$ is asymmetric provided that at least one of $F$ and $B$ is non-empty, so $D_K \geq |\mathcal F \times \mathcal B| -1$.

Similarly as above we have $|F^*| = f+1$ and $|B^*| \geq b-1$, so  $|\mathcal F \times \mathcal B| \geq f+b \geq \Delta - 2$.
In particular, since $\Delta \geq 5$ we get
\[D_K \geq |\mathcal F \times \mathcal B|-1 \geq \Delta - 3 \geq \lfloor  \Delta / 2 \rfloor \geq N_K,\] which completes the case $h=1$.

Finally, assume that $h \geq 2$.
Then by \ref{itm:star} there are at least $(h+1)$ vertices in $K$ which are incident to at least one red or green horizontal edge.
Let $F^*$ be the set of forward edges incident to those vertices, and let $B^*$ be the set of non-blue back edges incident to them.

Due to the colouring of the horizontal edges, every persistent automorphism which fixes $K$ setwise must also fix it pointwise, and thus no two different elements in $B^*$ can be mapped onto each other by a persistent automorphism.
Moreover, any decoration of $K$ is asymmetric, so $D_K\geq |\mathcal F \times \mathcal B|$

As above, clearly $|F^*| \geq f\cdot (h+1)$ and $|B^*| \geq (b-1)\cdot (h+1)$.
If $b = 1$, then we get
\[|\mathcal F \times \mathcal B| \geq f(h+1)+1 \geq f+h+1 = \Delta \geq N_K.\]
If $b > 1$, then as in the case $h=0$, we have $N_K \leq \Delta  - 1$.
This means that we get
$|\mathcal F \times \mathcal B| \geq (h+1)(f + b - 1) \geq f+b+h-1 = \Delta - 1 \geq N_K.$
Hence \[D_K \geq |\mathcal F \times \mathcal B| \geq N_K,\] thus completing the proof of Claim~\ref{clm:DgeqN}.

By the above discussion we can choose non-similar asymmetric decorations for all components of the graph  induced by the horizontal edges of $S_i$.
If $h = 1$, then the decorations can be chosen such that they contain only edges incident to one of the two vertices in each component.
If $h \geq 2$ we can make sure that the decoration does not contain any edges incident to a vertex all of whose horizontal edges are coloured blue.
Furthermore, we can make sure that one component in each orbit is decorated by $(\emptyset, \emptyset)$, unless $h=1$ and there is a persistent automorphism swapping the two endpoints of any component in the orbit.

Let $c_i$ be the colouring obtained from $c_{i-1}$ by colouring the horizontal edges of $S_i$ according to \ref{itm:cross0}--\ref{itm:cross2}, then choosing non-similar asymmetric decorations for each component of $H$ as described above, and decorating the components accordingly.
It remains to show that $c_i$  satisfies properties \ref{itm:rootcolour}--\ref{itm:blueedges}.

For property \ref{itm:rootcolour}, first observe that we did not change the colour of any edge incident to $r$.
In particular, all edges incident to $r$ are blue with respect to $c_i$. Further note that the only edges that are blue with respect to $c_i$, but not with respect to $c_{i-1}$ are horizontal and back edges incident to $S_i$.
Property \ref{itm:rootcolour} holds for $c_{i-1}$, hence the only way that $c_i$ could violate \ref{itm:rootcolour} is, if all edges incident to a vertex in $S_i$ or a back neighbour of $S_i$ are blue in $c_i$.
We will show that all such vertices have at least one incident edge coloured red or green.

First, let $x \in S_i$.
Recall that decorating assigns colour red or green to all forward edges, so in case $f > 0$ there is a non-blue edge incident to $x$.
Since \ref{itm:blueedges} holds for $c_{i-1}$, there is at most one back edge incident to $x$ coloured blue with respect to $c_{i-1}$.
Decorating changes the colours of at most two more back edges incident to $x$.
Thus, if $b \geq 4$, then $x$ is incident to at least one red or green back edge.
So we may assume that $f=0$ and $b \leq 3$, and thus $h = \Delta - f - b \geq 2$ because $\Delta \geq 5$.
If $x$ is incident to a non-blue horizontal edge, then there is nothing to show.
If all horizontal edges incident to $x$ are coloured blue, then our choice of decorations makes sure that no edges incident to $x$ are used in the decorations.
In particular, if all edges incident to $x$ are blue with respect to $c_i$, then all back edges incident to $x$ must be blue with respect to $c_{i-1}$.
Consequently, the only back neighbour of $x$ is $r$. Since $f = 0$, we conclude that $h = \Delta - 1$ which implies that $G$ is complete, contradicting one of our initial assumptions.

Next, let $x \neq r$ be a back neighbour of $S_i$.
By \ref{itm:blueedges}, all edges between $x$ and $S_i$ are coloured red or green with respect to $c_{i-1}$.
Note that if $x$ is incident to a vertex $y \in S_i$, then it must be incident to each vertex in the orbit of $y$ under persistent automorphisms.
If $h=1$ and there is a persistent automorphism swapping the two vertices of some component $K$ incident to $x$, then $x$ is incident to both vertices of $K$.
The decoration of $K$ only used edges incident to one of the two vertices, so the edge from $x$ to the other vertex has the same colour in $c_i$ as in $c_{i-1}$ whence $x$ is incident to a non-blue edge.
If $h \neq 1$, or $h=1$ and no persistent automorphism swaps the two vertices of a component incident to $x$, then $x$ is incident to a component $K$ with decoration $(\emptyset,\emptyset)$.
All edges connecting $x$ to $K$ have the same colour in $c_i$ as in $c_{i-1}$, again showing that $x$ must be incident to a non-blue edge.

To see that property \ref{itm:changecolour} holds for $c_i$, note that in \ref{itm:cross0}--\ref{itm:cross2} only horizontal edges incident to $S_i$ were recoloured, and that the decorating step only affects forward and back edges incident to $S_i$.
All other edges have the same colour with respect to $c_i$ and $c_{i-1}$.

The proof of \ref{itm:preservecolour} rests on the following claim.

\begin{clm}
\label{clm:preservecolour}
Let $j < i$, let $\gamma$ be an element of $\Aut(G,r)$ that preserves the restriction of $c_i$ to $\hat E_j$, and let $e \in \hat E_j$.
Then $e$ and $\gamma(e)$ have the same colour with respect to $c_{i-1}$.
\end{clm}

Using this claim, it is easy to complete the proof of \ref{itm:preservecolour}.
Indeed, if $j < i$, then any automorphism which preserves the restriction of $c_i$ to $\hat E_j$ must also preserve the restriction of $c_{i-1}$ to $\hat E_j$.
By \ref{itm:preservecolour} for $c_{i-1}$, this implies that any such automorphism has to fix $S_j$ pointwise.
For the case $j = i$, note that any automorphism preserving the restriction of $c_i$ to $\hat E_i$ must be persistent.
Indeed, any such automorphism must pointwise fix $S_j$ for $j < i$ because $\hat E_j \subseteq \hat E_i$, and it must preserve the colouring on the horizontal edges because $m(i) \geq i$.
The decorations on the components of $H$ were asymmetric and non-similar, so any persistent automorphism which preserves them must pointwise fix $S_i$.

It remains to prove Claim \ref{clm:preservecolour}.
Recall that the $S_i$ were defined as the orbits with respect to $\Aut(G,r)$.
In particular, an automorphism $\gamma$ as in Claim \ref{clm:preservecolour} must setwise preserve every $S_i$, and thus also every $E_i$.
In particular, if $e \in E_1$, then Claim \ref{clm:preservecolour} is true by \ref{itm:rootcolour}.
If $e \notin E_i$, then it is true by \ref{itm:changecolour}.
If $e$ is a forward or horizontal edge of $E_i$, then so is $\gamma(e)$, and by \ref{itm:allgreen} both of them are green with respect to $c_{i-1}$, so the claim also holds in this case.

If $e$ is a red or green back edge, then $e$ was not contained in any of the decorations used in the recolouring procedure, and thus $c_i(e) = c_{i-1}(e)$.
The edge $\gamma(e)$ in this case is also a red or green back edge, hence we get $c_i(\gamma(e)) = c_{i-1}(\gamma(e))$, and consequently $c_{i-1}(e) = c_{i-1}(\gamma(e))$, unless $e$ is a blue back edge.

Finally, consider the case that $e$ is a blue back edge.
In this case, $e$ must have been contained in one of the decorations.
By property \ref{itm:back} of decorations, if $e$ is red in $c_{i-1}$, then there are no other blue back edges connected to the same component of $H$ but not to $r$.
On the other hand, if $e$ is green in $c_{i-1}$, then there is exactly one more such blue back edge.
Note that if one back edge is contained in $\hat E_j$, then $m(j) \geq i$ and consequently all back edges are contained in $\hat E_j$.
As $\gamma$ preserves the restriction of $c_i$ to $\hat E_j$, it cannot map a component of $H$ with two blue (with respect to $c_i$) back edges not incident to $r$ to another component with only one such back edge.
In particular, it cannot map a back edge which is blue in $c_i$ and red in $c_{i-1}$ to a back edge which is blue in $c_i$ and green in $c_{i-1}$.
This completes the proof of Claim \ref{clm:preservecolour} and thus also of \ref{itm:preservecolour}.

Property \ref{itm:allgreen} for $c_i$ follows from \ref{itm:changecolour} combined with \ref{itm:allgreen} for $c_{i-1}$. Finally, property \ref{itm:blueedges} follows from \ref{itm:changecolour} together with the fact that decorating a component does not assign colour blue to any of the forward edges.
\end{proof}

%%%%%%%%%%%%%%%%%%%%%%%%%%%%%%%%%%%%%%%%%%%%%%%%%%%%%%%%%%%%%%%%
%%%%%%%%%%%%%%%%%%%%%%%%%%%%%%%%%%%%%%%%%%%%%%%%%%%%%%%%%%%%%%%%
\section{Graphs with infinite degrees}\label{infinite}

In this section we will consider the distinguishing index of $\kappa$-regular graphs in the case of infinite $\kappa$.
We prove that there are arbitrary large cardinals $\kappa$ with the property that every connected $\kappa$-regular graph has distinguishing index at most two.
In other words, for every cardinal $\gamma$, there exists a cardinal $\kappa>\gamma$ with said property.
For self-sufficiency of this paper, we will provide some basic definitions and well known facts about these cardinals.
For a detailed treatment of ordinal and cardinal numbers see for example~\cite{Jech}.
As usual in ZFC, we identify cardinal numbers with initial ordinals.

A \textit{cardinal number} is an ordinal that is not equinumerous with any smaller ordinal.
The least cardinal number that is greater than given cardinal $j$ is called the \textit{successor} of $j$ and is denoted by $j^+$.

The \textit{aleph hierarchy} assigns a cardinal number to every ordinal number.
It can be defined by transfinite induction as follows.
\begin{align*}
&\aleph_0 = |\mathbb{N}|,\text {the least infinite cardinal,}& \\
&\aleph_{\alpha+1} = \aleph_\alpha^+ \text{, for any ordinal }\alpha,& \\
&\aleph_\alpha = \sup \{ \aleph_\beta : \beta < \alpha \} \text{, for any limit ordinal }\alpha.&
\end{align*}

Every infinite cardinal number lies in the aleph hierarchy.
A cardinal number $\kappa$ is a fixed point of the aleph hierarchy if $\kappa=\aleph_\kappa$.

For every cardinal number $\lambda$ there is a greater cardinal with this property. Indeed, consider the sequence given by $\lambda_0=\lambda$, and $\lambda_{n+1}=\aleph_{\lambda_n}$, for $n \geq 0$.
Then $\kappa = \sup_{n \in \mathbb N} \lambda_n$ is a limit ordinal and $\aleph_\kappa = \sup_{n \in \mathbb N} \aleph_{\lambda_n} = \sup_{n \in \mathbb N} \lambda_{n+1} = \kappa$.
This fact was first noticed by Veblen in \cite{veb}.

Recall that every ordinal number $\alpha$ is the set of all ordinals smaller than $\alpha$.
Thus fixed points of the aleph hierarchy may be characterised as exactly these uncountable cardinals $\kappa$, such that the set of cardinals smaller than $\kappa$ is of cardinality $\kappa$.

\begin{thm}
\label{thm:aleph}
Let $G$ be connected  $\kappa-$regular graph and let $\kappa$ be a fixed point of aleph hierarchy i.e. $\kappa=\aleph_\kappa$. Then $D'(G)\leq 2$.
\end{thm}
\begin{proof}

Let $x_0, x_1, x_2 \dots, x_\alpha \dots$ for $\alpha<\kappa$ be a $\kappa$-enumeration of the vertices of $G$. We will colour the graph by transfinite induction on $i$.

Assume that for every ordinal $j<i$ vertex $x_j$ has exactly $\aleph_j$ incident blue edges, the rest of its incident edges are green, and the remaining edges are uncoloured. Notice that the only edges incident to $x_i$ that are already coloured are edges incident also to $\{x_j: j < i \}$. As $|\{x_j: j < i \}|\leq \aleph_i$, vertex $x_i$ is incident to at most $\aleph_i$ blue edges, and $\kappa$  edges incident to $x_i$ are still uncoloured. As $\aleph_i < \aleph_\kappa = \kappa$, we can choose $\aleph_i$ edges between $x_i$ and the vertices among $\{x_j: j>i \}$ and colour them blue, and we colour the remaining edges incident with $x_i$ with green.

After the induction, every vertex $x_i$ has exactly $\aleph_i$ incident blue edges because $x_i$ has $\aleph_i$ incident blue edges after the step $i$, the remaining incident edges are green and we do not recolour any edge. As every vertex is incident to a different number of blue edges it is fixed by every colour preserving automorphism and therefore we obtained a distinguishing edge-colouring of $G$.
\end{proof}

A simple corollary of this theorem is that the class of cardinals $\kappa$ such that any connected $\kappa$-regular graph admits a distinguishing edge-colouring with $2$ colours is unbounded (equivalently it is a proper class).
We conjecture that every infinite $\kappa$ has this property.

\begin{conj}
Let $G$ be a connected $\kappa$-regular graph for some infinite cardinal $\kappa$. Then $G$ admits a distinguishing edge-colouring with two colours.
\end{conj}

Although Theorem \ref{thm:aleph} does not give us results about every cardinal it can be used to obtain consistency results about some of them. We say that a cardinal number $\kappa$ is a \textit{regular cardinal} if the sum of less than $\kappa$ sets of cardinality less than $\kappa$ has cardinality less than $\kappa$. Every cardinal $j^+$ is a regular cardinal as well as the cardinal $\aleph_0$. If a cardinal $\kappa$ is a regular infinite cardinal, then it is consistent with ZFC that the cardinal $2^\kappa$ is a fixed point of the aleph hierarchy. This was first noticed by Solovay \cite{Sol}, for proofs see \textit{Application of Forcing} Chapter in Jech \cite{Jech}. Summarising, we obtained the following consistency result, which may be interesting even in the case when $\kappa=\aleph_0$.

\begin{thm}
\label{thm:aleph2}Let $\kappa$ be a regular infinite cardinal. Then it is consistent with ZFC that every connected $2^\kappa-$regular graph has distinguishing index at most two.
\end{thm}

%%%%%%%%%%%%%%%%%%%%%%%%%%%%%%%%%%%%%%%%%%%%%%%%%%%%%%%%%%%%%%%
%\bibliographystyle{abbrv}

%\bibliography{lit.bib}
%%%%%%%%%%%%%%%%%%%%%%%%%%%%%%%%%%%%%%%%%%%%%%%%%%%%%%%%%%%%%%%

\end{document}